\documentclass[a4paper,12pt]{article}

\usepackage[textwidth=125mm, textheight=195mm]{geometry}
\usepackage[english]{babel}
\usepackage{graphicx}
\usepackage{amsmath}
\usepackage{amsfonts}
\usepackage{amsthm}
\usepackage{epstopdf}
\usepackage{color}

\begin{document}

\newcommand{\wk}{\mbox{$\,<$\hspace{-5pt}\footnotesize )$\,$}}

\numberwithin{equation}{section}
\newtheorem{teo}{Theorem}
\newtheorem{lemma}{Lemma}

\newtheorem{coro}{Corollary}
\newtheorem{prop}{Proposition}
\theoremstyle{definition}
\newtheorem{definition}{Definition}
\theoremstyle{remark}
\newtheorem{remark}{Remark}

\newtheorem{scho}{Scholium}
\newtheorem{open}{Question}
\newtheorem{example}{Example}
\numberwithin{example}{section}
\numberwithin{lemma}{section}
\numberwithin{prop}{section}
\numberwithin{teo}{section}
\numberwithin{definition}{section}
\numberwithin{coro}{section}
\numberwithin{figure}{section}
\numberwithin{remark}{section}
\numberwithin{scho}{section}

\bibliographystyle{abbrv}

\title{A new geometric viewpoint on Sturm-Liouville eigenvalue problems}
\date{}

\author{Vitor Balestro  \\ Instituto de Matem\'{a}tica e Estat\'{i}stica \\ Universidade Federal Fluminense \\ 24210201 Niter\'{o}i \\ Brazil \\ vitorbalestro@id.uff.br \and Horst Martini \\ Fakult\"{a}t f\"{u}r Mathematik \\ Technische Universit\"{a}t Chemnitz \\ 09107 Chemnitz\\ Germany \\ martini@mathematik.tu-chemnitz.de \and Ralph Teixeira \\ Instituto de Matem\'{a}tica e Estat\'{i}stica \\ Universidade Federal Fluminense \\ 24210201 Niter\'{o}i \\ Brazil \\ ralph@mat.uff.br}

\maketitle

\begin{abstract}
In Euclidean plane geometry, \emph{cycloids} are curves which are homothetic to their respective bi-evolutes. In smooth normed planes, cycloids can be similarly defined, and they are characterized by their radius of curvature functions being solutions to eigenvalue problems of certain Sturm-Liouville equations. In this paper, we show that, for the eigenvalue $\lambda = 1$, this equation is a previously studied Hill equation which is closely related to the geometry given by the norm. We also investigate which geometric properties can be derived from this equation. Moreover, we prove that if the considered norm is defined by a Radon curve, then the solutions to the Hill equation are given by trigonometric functions. Further, we give conditions under which a given Hill equation induces a planar Minkowski geometry, and we prove that in this case the geometry is Euclidean if an eigenvalue other than $\lambda =1$ induces a reparametrization of the original unit circle.

\end{abstract}


\bigskip


\section{Introduction}

In Euclidean plane geometry, the \emph{evolute} of a smooth curve is defined  to be the locus of its centers of curvature. This can be extended to any smooth normed plane (= Minkowski plane) by using the concept of \emph{circular curvature} (see \cite{Ba-Ma-Sho}), which is a natural extension of the usual curvature. A \emph{cycloid} is a curve which is homothetic to its \emph{bi-evolute}, that is, the evolute of the evolute. This can also be extended to normed planes, but the bi-evolute is then calculated with respect to the \emph{anti-norm}. This approach was taken in \cite{Cra-Tei-Ba1}, where the classical Sturm-Liouville theory was used to investigate the existence and the properties of closed cycloids (see also \cite{Cra-Tei-Ba2} for the discrete version).

In this paper, we show that the Sturm-Liouville equation of the Minkowskian cycloids with eigenvalue $\lambda = 1$ is precisely the Hill equation studied in \cite{Pet-Bar} from the viewpoint of the geometry of normed planes. Adopting a different approach, we extend the investigations on the relations between this equation and the geometry of the associated Minkowski plane.

In Section \ref{basic}, we outline the basic concepts of planar Minkowski geometry that will be used throughout the text. For basic references on the topic, we refer the reader to the book \cite{Tho} and the surveys \cite{martini2} and \cite{martini1}. As references for the concepts of orthogonality and curvature in normed planes, \cite{alonso} and \cite{Ba-Ma-Sho} should be mentioned; see also \cite{Pet} for curvature notions and \cite{Ma-Wu} for general curve theory in normed planes. In Section \ref{evocycl} we follow \cite{Cra-Tei-Ba1} to define evolutes, bi-evolutes and cycloids in normed planes. We also give a new characterization of closed cycloids, and relate that to the eigenvalues obtained in the mentioned paper.

We prove in Section \ref{radoncycl} that if the norm is Radon, then the cycloids associated to the eigenvalue $\lambda = 1$ are given by the trigonometric functions for normed planes studied in \cite{Ba-Ma-Tei1} and \cite{Ba-Sho}. This is similar to what happens when the geometry is Euclidean, and it is not true if the norm is not Radon.

The last two sections are devoted to the study of the mentioned Hill equation and its relations to the geometry of the considered plane. In Section \ref{hillsec}, we obtain this equation as a particular case of the Minkowskian cycloids (Sturm-Liouville) equation (namely, the case where $\lambda= 1$). We prove that the Hill equation associated to a given normed plane is ``intrinsically" determined by the geometry of the normed plane, meaning that this is invariant under (non-degenerate) affine equivalence. Also, we derive geometric properties of the Minkowski plane from the solutions of the differential equation, such as characterizations of the cases when the norm is Radon or Euclidean. Finally, in Section \ref{condhill} we investigate under what conditions a given Hill equation is associated to the geometry of some Minkowski plane. We also prove that if the solutions associated to eigenvalues of the Sturm-Liouville equation other than $\lambda = 1$ are linear reparametrizations of the original solutions, then the associated geometry is Euclidean.

For the part of the standard theory of ordinary differential equations used here (essentially results on existence and uniqueness of solutions) we refer to the book \cite{coddington}.

\section{Basic theory}\label{basic}

Throughout this text, we work with a normed plane $(\mathbb{R}^2,||\cdot||)$ whose \emph{unit ball} is the set $B:=\{x \in \mathbb{R}^2:||x|| \leq 1\}$. The boundary $\partial B:=\{x\in\mathbb{R}^2:||x|| = 1\}$ of $B$ is called the \emph{unit circle}. We will always assume that $(\mathbb{R}^2,||\cdot||)$ is \emph{strictly convex} and \emph{smooth}. The first condition means that the unit circle $\partial B$ does not contain a line segment, and the second means that $\partial B$ does not have \emph{singular points}, which are points at which the unit ball is supported by more than one line.

We let $\omega:\mathbb{R}^2\times\mathbb{R}^2\rightarrow\mathbb{R}$ be a fixed non-degenerate, skew-symmetric bilinear form (that is, a \emph{symplectic form}). This has to be a scalar multiple of the usual determinant (see \cite{Ma-Swa} for details), and later it will be clear why we do not simply use the standard determinant. Such a symplectic form yields a natural isomorphism between $\mathbb{R}^2$ and its dual space $(\mathbb{R}^2)^*=\mathbb{R}_2$ by contraction in the first variable:
\begin{align*} \mathbb{R}^2\ni x \mapsto \iota_x\omega(\cdot) := \omega(x,\cdot):\mathbb{R}^2\rightarrow\mathbb{R}.
\end{align*}
Using this isomorphism, we may identify the dual norm, defined as usual in $\mathbb{R}_2$, with a norm in $\mathbb{R}^2$, which we call the \emph{anti-norm}:
\begin{align*} ||x||_a := \sup\{\omega(x,y):y \in \partial B\},
\end{align*}
for each $x \in \mathbb{R}^2$. It is not difficult to see that $||\cdot||_a$ is indeed a norm on $\mathbb{R}^2$. Also, the anti-norm is dual to the norm, in the sense that the anti-norm of the anti-norm (with respect to the same fixed symplectic form) is the original norm. For respective proofs we cite \cite{Ma-Swa}. The unit ball and the unit circle in the anti-norm will be called \emph{unit anti-ball} (denoted by $B_a$) and \emph{unit anti-circle} (denoted by $\partial B_a$), respectively.

The importance of the anti-norm comes also from its relation to Birkhoff orthogonality. Given two vectors $x,y \in \mathbb{R}^2$, we say that $x$ is \emph{Birkhoff orthogonal} to $y$ (denoted by $x \dashv_B y$) if $||x|| \leq ||x+ty||$ for any $t \in \mathbb{R}$. Geometrically, $x\dashv_B y$ means that if $x$ and $y$ are non-zero vectors, a line in the direction of $y$ supports the unit ball at $x/||x||$. From this geometric viewpoint, it is clear that Birkhoff orthogonality is not a symmetric relation (for more on Birkhoff orthogonality, and also on other orthogonality concepts in normed spaces, \cite{alonso} should be consulted). In what follows, we denote by $\dashv_B^{\,a}$ the \emph{Birkhoff orthogonality in the anti-norm}.
\begin{prop} For any $x,y \in \mathbb{R}^2$ we have the following inequality:
\begin{align*} |\omega(x,y)| \leq ||x||\cdot||y||_a,
\end{align*}
and equality holds if and only if $x \dashv_B y$. Moreover, the anti-norm reverses Birkhoff orthogonality, that is, $x \dashv_B y$ if and only if $y \dashv_B^{\,a} x$ (see Figure \ref{birkhoff}).
\end{prop}

\begin{figure}[h]
\centering
\includegraphics{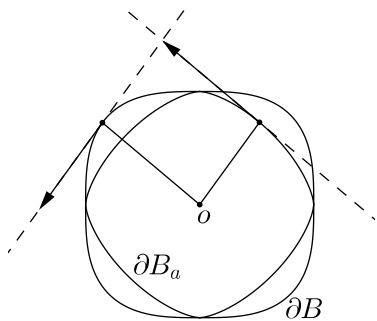}
\caption{The anti-norm reverses Birkhoff orthogonality.}
\label{birkhoff}
\end{figure}

Again, we refer the reader to \cite{Ma-Swa} for a proof, and also for related discussions. A normed plane is called a \emph{Radon plane} if Birkhoff orthogonality is symmetric or, equivalently, if the anti-norm is a multiple of the norm. In this case, up to re-scaling the fixed symplectic form, we can assume that the norm and the anti-norm are equal, and we will always assume that this is done. The unit circle of a Radon plane is called a \emph{Radon curve}.

\begin{remark} The most usual identification between $\mathbb{R}^2$ and its dual is given by the standard inner product $\langle\cdot\,,\cdot\rangle:\mathbb{R}^2\times\mathbb{R}^2\rightarrow\mathbb{R}$. With this identification, the unit ball of the dual norm is identified with the polar body of $\partial B$. The isomorphism given by the usual determinant is, geometrically, a $\pi/2$-rotation of the identification given by the standard inner product. Therefore, a centrally symmetric curve is Radon if and only if it is homothetic to a $\pi/2$-rotation of its polar curve.
\end{remark}

We denote by $\mathbb{R}^2_*$ the set of non-zero vectors of $\mathbb{R}^2$. In a given normed plane $(\mathbb{R}^2,||\cdot||)$, we define the \emph{Minkowskian sine function} $\mathrm{sm}:\mathbb{R}^2_*\times\mathbb{R}^2_*\rightarrow\mathbb{R}$ as
\begin{align*} \mathrm{sm}(x,y) = \frac{\omega(x,y)}{||x||\cdot||y||_a}.
\end{align*}
Notice that this is regarded as a (not necessarily skew-symmetric) map of two unit vectors. Geometrically, the Minkowskian sine function $\mathrm{sm}(x,y)$ is the (signed) distance in the norm from the origin to the line $\mathbb{R} \ni t \mapsto x+ty$.
\begin{prop} For any $x,y \in \partial B$, we have the equality
\begin{align*} |\mathrm{sm}(x,y)| = \inf\{||x+ty||:t \in \mathbb{R}\},
\end{align*}
and $|\mathrm{sm}(x,y)| = 1$ if and only if $x \dashv_B y$. Moreover, the Minkowskian sine function is skew-symmetric if and only if the normed plane is a Radon plane.
\end{prop}
For the proof, and also for related discussions involving the Minkowskian sine function, we refer to \cite{Ba-Ma-Tei1}. Notice that the Minkowskian sine function can be defined also for normed planes which are not smooth or strictly convex. But this is not the case for our next task, that is, to introduce an extension of the elementary cosine function (for that purpose we have to assume smoothness).

If $(\mathbb{R}^2,||\cdot||)$ is a smooth normed plane, then there is a map $b:\mathbb{R}^2_*\rightarrow\mathbb{R}^2_*$ which associates each $x \in \mathbb{R}^2_*$ to the unique vector $b(x)$ such that
\begin{align*} x \dashv_B b(x), \ \mathrm{and} \\
\omega(x,b(x)) = ||x||.
\end{align*}
Notice that the equality above implies that $b(x) \in \partial B_a$, and that the basis $\{x,b(x)\}$ is positively oriented. The \emph{Minkowskian cosine function} $\mathrm{cm}:\mathbb{R}^2_*\times\mathbb{R}^2_*\rightarrow\mathbb{R}$ is defined as
\begin{align*} \mathrm{cm}(x,y) = \frac{\omega(y,b(x))}{||y||}.
\end{align*}
Here we can also give a geometric interpretation: for given $x,y \in \partial B$, we have that $\mathrm{cm}(x,y)$ is the (signed) value of the distance to the origin from the parallel through $y$ of the line which supports the unit circle at $x$. For more on the Minkowskian cosine function, the reader should consult \cite{Ba-Sho}.
In the case where $\omega$ is the usual determinant, it is easy to see that the Minkowskian sine and cosine functions are extensions of the elementary sine and cosine functions of Euclidean geometry. It is also clear that we have the equality
\begin{align*} \mathrm{cm}(x,y) = \mathrm{sm}(y,b(x)),
\end{align*}
for any $x,y \in \mathbb{R}^2_*$. Moreover, the Minkowskian sine and cosine functions are \emph{positively homogeneous}. This means that $\mathrm{sm}(\alpha x,\beta y) = \mathrm{sm}(x,y)$ for any $x,y \in \mathbb{R}^2_*$ and $\alpha,\beta> 0$, and the same holds for the Minkowskian cosine function. This means that these functions can be regarded as defined for oriented directions rather than for vectors.

From now on, beyond assuming that $(\mathbb{R}^2,||\cdot||)$ is strictly convex, we \textbf{always} assume that the unit circle $\partial B$ is also a smooth curve, meaning that it admits a parametrization which has derivatives of all orders (we often need less regularity, but we demand smoothness for the sake of simplicity), and whose Euclidean curvature does not vanish. The \emph{Minkowskian length} (which we will call simply \emph{length}) of a curve $\gamma:[a,b]\rightarrow(\mathbb{R}^2,||\cdot||)$ can be defined by polygonal approximations, as usual. If $\gamma$ is smooth, then one can prove that this definition is equivalent to
\begin{align*} l(\gamma) = \int_a^b||\gamma'(s)|| \, ds.
\end{align*}
Since $\partial B$ is smooth (as a curve) and regular, it admits an arc-length parametrization $\varphi(t):\mathbb{R}\,\mathrm{mod}\,l(\partial B)\rightarrow \mathbb{R}^2$. Let $\gamma(s):[0,l(\gamma)]\rightarrow\mathbb{R}^2$ be a smooth curve parametrized by arc-length. We define the function $t(s):[0,l(\gamma)]\rightarrow\mathbb{R}\,\mathrm{mod}\,l(\partial B)$ intrinsically by the equality
\begin{align*} \gamma'(s) = \frac{d\varphi}{dt}(t(s)).
\end{align*}
Geometrically, for each $s \in [0,l(\gamma)]$ the number $t(s) \in \mathbb{R}\,\mathrm{mod}\,l(\partial B)$ is the value of the parameter for which $\gamma'(s)$ is the unit tangent vector to $\partial B$ at $\varphi(t(s))$. We define the \emph{circular curvature} (or simply \emph{curvature}) of $\gamma$ at $\gamma(s)$ to be the number
\begin{align*} k(s) := t'(s),
\end{align*}
for each $s \in [0,l(\gamma)]$. Intuitively, the curvature measures ``how quickly" the unit tangent vector field of $\gamma$ rotates when identified as a tangent vector field of the unit circle. For concepts of curvature in normed planes, we mention once more the expository paper \cite{Ba-Ma-Sho}. The importance of this particular curvature concept for our purposes comes from its contact interpretation: if $k(s) \neq 0$, then $k(s)^{-1}$ is the radius of an osculating circle attached to $\gamma$ at $\gamma(s)$. Indeed, we can notice that the derivative of the reparametrization $\varphi(t(s)):[0,l(\gamma)]\rightarrow\mathbb{R}^2$ at a given fixed point $s_0 \in [0,l(\gamma)]$, where the curvature does not vanish, is given as
\begin{align*} \varphi'(s_0) = t'(s_0)\cdot\frac{d\varphi}{dt}(t(s_0)) = k(s_0)\cdot\gamma'(s_0),
\end{align*}
and hence the Minkowski circle, parametrized (at least locally) by
\begin{align*} s \mapsto \gamma(s_0) - k(s_0)^{-1}\cdot\varphi(s_0) + k(s_0)^{-1}\cdot\varphi(s)\,,
\end{align*}
has second order contact with $\gamma$ at $\gamma(s_0)$. For that reason, the number $\rho(s) := k(s)^{-1}$ is called the \emph{radius of curvature} of $\gamma$ at $\gamma(s)$. The point $c(s) := \gamma(s) - \rho(s)\varphi(s)$ is called the \emph{center of curvature} of $\gamma$ at $\gamma(s)$, as it is illustrated in Figure \ref{osculating}.

\begin{figure}[h]
\centering
\includegraphics{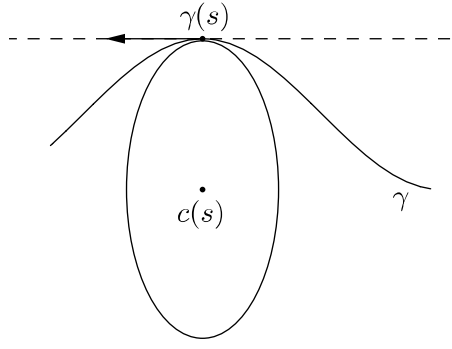}
\caption{The (Minkowskian) osculating circle of $\gamma$ at $\gamma(s)$.}
\label{osculating}
\end{figure}

This interpretation will be particularly important for us, because it allows to re-obtain the curvature in a way which is independent of the fixed parametrizations of the curve and of the unit circle. We let $\varphi(\theta)$ be any smooth regular parametrization of $\partial B$, and $\gamma$ be a smooth curve. We endow $\gamma$ with a parametrization $\gamma(\theta)$ such that, for each $\theta$, the tangent vector $\gamma'(\theta)$ is a non-negative multiple of $\varphi'(\theta)$. In other words, we have a smooth non-negative function $\rho(\theta)$ such that
\begin{align*} \gamma'(\theta) = \rho(\theta)\cdot\varphi'(\theta).
\end{align*}
It follows that $\rho(\theta)$ is precisely the radius of curvature of $\gamma$ at $\gamma(\theta)$, and hence its inverse is the curvature (in the points where $\rho$ does not vanish, of course). A proof is presented in \cite{Ba-Ma-Sho}.

\section{Evolutes and cycloids}\label{evocycl}

Recall that we are assuming that $(\mathbb{R}^2,||\cdot||)$ is a smooth and strictly convex normed plane, and that its unit circle $\partial B$ admits a smooth parametrization by arc-length which we denote by $\varphi(t):\mathbb{R}\,\mathrm{mod}\,l(\partial B)\rightarrow\mathbb{R}^2$. Since it is clear that $\varphi(t) \dashv_B \varphi'(t)$, we have that $\omega(\varphi(t),\varphi'(t)) = ||\varphi'(t)||_a$. Hence we can parametrize the unit anti-circle $\partial B_a$ by the map $\psi(t):\mathbb{R}\,\mathrm{mod}\,l(\partial B)\rightarrow \mathbb{R}^2$ defined by
\begin{align*} \psi(t) = \frac{\varphi'(t)}{\omega(\varphi(t),\varphi'(t))}.
\end{align*}
From our hypothesis on $(\mathbb{R}^2,||\cdot||)$ and $\partial B$, it follows immediately that this is a smooth regular parametrization. Notice that, for each $t$, we have $\psi(t) \dashv_B^{\,a}\psi'(t)$, and since the anti-norm reverses Birkhoff orthogonality, we get $\psi'(t) \dashv_B \psi(t)$. This leads to the equality
\begin{align*} \omega(\psi(t),\psi'(t)) = ||\psi'(t)||,
\end{align*}
for each $t \in \mathbb{R}\,\mathrm{mod}\,l(\partial B)$. Also, from uniqueness of Birkhoff orthogonality, we have that $\psi'(t)$ is in the direction of $\varphi(t)$. Hence $\psi$ can be regarded as a dual parametrization of $\varphi$, since we obviously have also the equality $\omega(\varphi(t),\varphi'(t)) = ||\varphi'(t)||_a$. 

\begin{remark}Recall that we are assuming that the Euclidean curvature of $\varphi$ does not vanish, from where we get $\varphi'' \neq 0$, and hence $\psi'\neq 0$. Notice that this condition can be replaced by demanding straightforwardly that $\varphi''$ does not vanish for an arc-length parametrization $\varphi$ of the unit circle. 
\end{remark}

Let $\gamma$ be a smooth regular curve. As at the end of the previous section, we consider a parametrization $\gamma(t)$ such that the tangent vector $\gamma'(t)$ is always a multiple of $\varphi'(t)$, but not necessarily non-negative. In other words, we have
\begin{align*} \gamma'(t) = \rho(t)\cdot\varphi'(t),
\end{align*}
for each $t$, where $\rho(t) \in [0,+\infty)$ is the (signed) radius of curvature of $\gamma$ at $\gamma(t)$. In the points where $\rho(t) \neq 0$, we have that $\rho(t)^{-1}$ is the curvature of $\gamma$ at $\gamma(t)$ for any $t$. The \emph{evolute} of $\gamma$ is the curve defined as
\begin{align*} \xi(t) := \gamma(t) - \rho(t)\cdot\varphi(t).
\end{align*}
Geometrically, the evolute of a curve is the locus of its centers of curvature (except for the points where $\rho(t) = 0$). In the next lemma we show that it is easy to calculate the curvature of the evolute in the geometry induced by the anti-norm.
\begin{lemma} The tangent vector $\xi'(t)$ of the evolute is always parallel to $\psi'(t)$. Also, the radius of curvature of the evolute in the anti-norm is given by
\begin{align*} \delta(t) := \frac{\rho'(t)}{\omega(\psi(t),\psi'(t))},
\end{align*}
for any $t$.
\end{lemma}
\begin{proof} Differentiating $\xi(t)$, we obtain
\begin{align}\label{antiradius} \xi'(t) = \gamma'(t) - \rho(t)\cdot\varphi'(t) - \rho'(t)\cdot\varphi(t) = -\rho'(t)\cdot\varphi(t).
\end{align}
On the other hand, since $\psi'(t)$ points in the same direction as $\varphi(t)$, we may write $\varphi(t) = \alpha(t)\cdot\psi'(t)$ for some function $\alpha$. Hence
\begin{align*}  \alpha(t)\cdot\omega(\psi(t),\psi'(t)) =  \omega(\psi(t),\varphi(t)) = \frac{\omega(\varphi'(t),\varphi(t))}{\omega(\varphi(t),\varphi'(t))} = -1,
\end{align*}
from which we obtain
\begin{align*} \varphi(t) = -\frac{\psi'(t)}{\omega(\psi(t),\psi'(t))}.
\end{align*}
Substituting this equality in (\ref{antiradius}), we finally get
\begin{align*} \xi'(t) = \frac{\rho'(t)}{\omega(\psi(t),\psi'(t))}\cdot\psi'(t).
\end{align*}
Therefore, we have indeed that $\delta(t)$ is the radius of curvature of $\xi(t)$ in the geometry given by the anti-norm (recall that $\psi(t)$ is a smooth regular parametrization of the unit anti-circle).

\end{proof}

Since now we have an expression for the radius of curvature of $\xi$ in the anti-norm, we can naturally define the evolute of $\xi$ in the anti-norm as
\begin{align*} \eta(t) = \xi(t) - \frac{\rho'(t)}{\omega(\psi(t),\psi'(t))}\cdot \psi(t).
\end{align*}
This is clearly the locus of the centers of the osculating anti-circles of $\xi$. The curve $\eta$ is called the \emph{bi-evolute} of $\gamma$. By duality, we obtain immediately that $\eta'(t)$ is parallel to $\varphi'(t)$ for each $t$, and that the radius of curvature of $\eta$ in the original norm is given by
\begin{align*} \rho_{\eta}(t) := -\frac{1}{\omega(\varphi(t),\varphi'(t))}\cdot\left(\frac{\rho'(t)}{\omega(\psi(t),\psi'(t))}\right)'.
\end{align*}

A \emph{Minkowskian cycloid} is a smooth curve which is homothetic to its bi-evolute. From the previous discussion, it is clear that the Minkowskian cycloids are precisely the solutions of the following Sturm-Liouville differential equation, whose variable is $u$:
\begin{align}\label{cycleq} \frac{1}{\omega(\varphi(t),\varphi'(t))}\cdot\left(\frac{u'(t)}{\omega(\psi(t),\psi'(t))}\right)' = -\lambda \cdot u(t)
\end{align}
for some number $\lambda > 0$. Indeed, if $\rho(t)$ is a solution of the differential equation above, then we simply define the curve $\gamma$ to be obtained from the equality $\gamma'(t) = \rho(t)\cdot\varphi'(t)$ by integration, and this is clearly a Minkowskian cycloid. Notice that this is unique up to translation.

Now we investigate the conditions for the solutions of (\ref{cycleq}) under which the \-cor\-res\-pon\-ding\ cycloid is a closed curve from a viewpoint slightly different to that used in \cite{Cra-Tei-Ba1}. This approach is somehow the inspiration for Theorem \ref{condition}. First of all, from now on we adopt, for the sake of convenience, the convention
\begin{align*} \pi := \frac{l(\partial B)}{2}.
\end{align*}

\begin{prop} For any $\pi$-periodic solution $\rho$ of (\ref{cycleq}), with a given $\lambda \in\mathbb{R}$, any associated cycloid is a closed curve.
\end{prop}
\begin{proof} An associated cycloid is such that $\gamma'(t) = \rho(t)\cdot\varphi'(t)$, and hence it is given by integration as
\begin{align*} \gamma(t) = \gamma(0) + \int_0^t\rho(s)\cdot\varphi'(s)\,ds.
\end{align*}
Therefore, we simply calculate
\begin{align*} \gamma(t+2\pi) = \gamma(0) + \int_0^{t+2\pi}\rho(s)\cdot\varphi'(s)\,ds = \\= \gamma(0) + \int_0^{t}\rho(s)\cdot\varphi'(s)\,ds + \int_{t}^{t+\pi}\rho(s)\cdot\varphi'(s)\,ds + \int_{t+\pi}^{t+2\pi}\rho(s)\cdot\varphi'(s)\,ds = \\ = \gamma(t) + \int_t^{t+\pi}\rho(s)\cdot\varphi'(s)\,ds + \int_t^{t+\pi}\rho(s+\pi)\cdot\varphi'(s+\pi)\,ds = \\ = \gamma(t) + \int_t^{t+\pi}\rho(s)\cdot\varphi'(s)\,ds - \int_t^{t+\pi}\rho(s)\cdot\varphi'(s)\,ds = \gamma(t),
\end{align*}
where the last equality is justified since $\rho$ is $\pi$-periodic and $\varphi'$ is $\pi$-antiperiodic.

\end{proof}

\begin{prop} Let $\rho$ be a solution of (\ref{cycleq}) for $\lambda \neq 1$. If $\rho(0) = -\rho(\pi)$ and $\rho'(0) = -\rho'(\pi)$, then the associated cycloid is closed.
\end{prop}
\begin{proof} The boundary conditions, together with existence and uniqueness of solutions, guarantee that $\rho$ is $2\pi$-periodic. We define the \emph{Minkowskian support function} of $\gamma$ to be $h_{\gamma}(t) = \omega(\gamma(t),\psi(t))$. Geometrically, this is the (signed) Minkowskian distance from $\gamma(t)$ to the origin $o$ (see Figure \ref{support}). It is easy to see that if $\gamma'(t) = \rho(t)\cdot\varphi'(t)$ then, up to a translation, $\gamma$ is given as
\begin{align*} \gamma(t) = h_{\gamma}(t)\cdot\varphi(t) + \frac{h'_{\gamma}(t)}{\omega(\psi(t),\psi'(t))}\cdot\psi(t).
\end{align*}
Consequently, it is sufficient to prove that $h_{\gamma}$ is $2\pi$-periodic. Differentiating the equality above, we have that
\begin{align*} \gamma' = \left(h_{\gamma} + \frac{1}{\omega(\varphi,\varphi')}\cdot\left(\frac{h_{\gamma}'}{\omega(\psi,\psi')}\right)'\right)\varphi',
\end{align*}
where we omitted the variable, for the sake of simplicity of notation. It follows that the radius of curvature is related to the Minkowskian support function by
\begin{align*} \rho = h_{\gamma} + \frac{1}{\omega(\varphi,\varphi')}\cdot\left(\frac{h_{\gamma}'}{\omega(\psi,\psi')}\right)'.
\end{align*}
Solving this equation in the variable $h_{\gamma}$, where we regard $\rho$ as a fixed solution of (\ref{cycleq}) with $\lambda \neq 1$, we have that
\begin{align*} h_{\gamma}(t) = \frac{\rho(t)}{1-\lambda}
\end{align*}
is a solution. Since this solution is unique up to initial conditions $h_{\gamma}(0)$ and $h'_{\gamma}(0)$, we get that $h_{\gamma}$ is indeed the Minkowskian support function of $\gamma$, up to a translation. Namely, we must have
\begin{align*} \gamma(0) = \frac{\rho(0)}{1-\lambda}\cdot\varphi(0) + \frac{\rho'(0)}{(1-\lambda)\omega(\varphi'(0),\psi'(0))}\cdot\varphi'(0).
\end{align*}
It follows that the Minkowskian support function of $\gamma$ is $2\pi$-periodic, and hence $\gamma$ is a closed curve.

\end{proof}

\begin{figure}[h]
\centering
\includegraphics{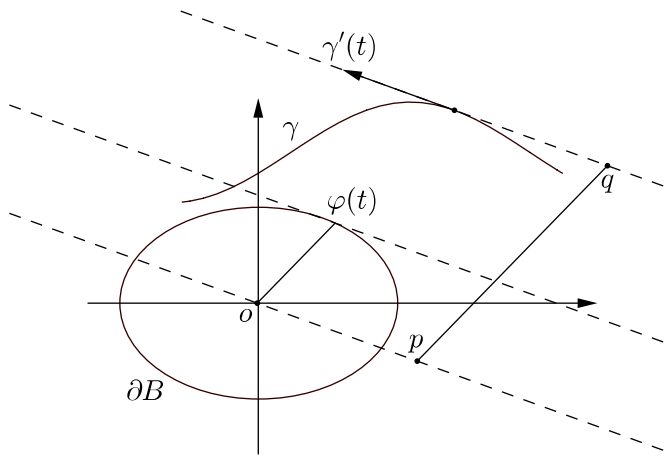}
\caption{$h_{\gamma}(t)$ is the (oriented) length of the segment $\mathrm{seg}(p,q)$.}
\label{support}
\end{figure}

\begin{remark} Notice that closed cycloids appear when $\lambda$ is such that the equation (\ref{cycleq}) has a $\pi$-periodic or $\pi$-antiperiodic (in the case where $\lambda \neq 1$) solution. In this case, we call $\lambda$ an \emph{eivengalue}. In \cite{Cra-Tei-Ba1}, classical Sturm-Liouville analysis was used to ensure that eigenvalues do indeed exist, and can be ordered in a non-decreasing sequence:
\begin{align*} \lambda_0 = 0 < \lambda_1^1=\lambda_1^2 = 1 < \lambda_2^1\leq \lambda_2^2 < \lambda_3^1\leq \lambda_3^2<\ldots <\lambda_k^1\leq\lambda_k^2<\ldots\,.
\end{align*}
From Theorem 4.1 of the mentioned paper, it follows that the eigenvalues with $k$ even correspond to $\pi$-periodic solutions, and the eigenvalues with $k$ odd give $\pi$-antiperiodic solutions.
\end{remark}

\section{Cycloids in Radon planes}\label{radoncycl}

In this section we aim to show that, in the particular case of Radon planes, the cycloids obtained with $\lambda = 1$ are given by trigonometric functions (as in the Euclidean case). We assume that $(\mathbb{R}^2,||\cdot||)$ is a Radon plane, endowed with a symplectic form $\omega$ for which the associated anti-norm $||\cdot||_a$ is equal to the original norm. We also still assume that $\varphi(t):\mathbb{R}\,\mathrm{mod}\,l(\partial B)\rightarrow \mathbb{R}^2$ is a smooth arc-length parametrization of the unit circle, and that
\begin{align*} \psi(t) = \frac{\varphi'(t)}{\omega(\varphi(t),\varphi'(t))}
\end{align*}
is the associated dual parametrization of the unit anti-circle $\partial B_a$, which is simply the unit circle $\partial B$. Under these hypotheses, notice that
\begin{align*} \omega(\varphi(t),\varphi'(t)) = ||\varphi(t)||\cdot||\varphi'(t)||_a = ||\varphi'(t)|| = 1,
\end{align*}
and consequently we have $\psi(t) = \varphi'(t)$. Notice very carefully that $\psi(t)$ may not be an arc-length parametrization of the unit circle (this is indeed the case if the norm is not Euclidean, see \cite{Ba-Sho} for a proof). For that reason, we define the function $\delta:\mathbb{R}\,\mathrm{mod}\,l(\partial B)\rightarrow \mathbb{R}^2$ by
\begin{align*} \delta(t) := ||\varphi''(t)||.
\end{align*}
Also, another expression for $\delta$ is obtained by the equality
\begin{align*} \omega(\psi(t),\psi'(t)) = \omega(\varphi'(t),\varphi''(t)) = ||\varphi'(t)||\cdot||\varphi''(t)|| = \delta(t),
\end{align*}
for each $t \in \mathbb{R}\,\mathrm{mod}\,l(\partial B)$. Therefore, in a Radon plane, the Minkowskian cycloids equation (\ref{cycleq}) becomes
\begin{align}\label{cyclradon} \left(\frac{u'}{\delta}\right)' = -\lambda\cdot u,
\end{align}
for $\lambda > 0$. In the Euclidean case, the equation above reads $u'' = -\lambda\cdot u$, and its solutions are constructed from the standard sine and cosine functions. In the general case above, standard theory of ordinary differential equations guarantees the existence of solutions, and uniqueness under given initial conditions $u(0)$ and $u'(0)$. We refer the reader to \cite{coddington} for a good reference on the theory of ordinary differential equations.

Our main theorems state that the solutions of the Minkowskian cycloids equation in a Radon plane can be obtained from the Minkowskian trigonometric functions defined in Section \ref{basic}, in a similar way as in the Euclidean case. We do it first for the case where $\lambda = 1$.

\begin{teo}\label{lambda1} If $\lambda = 1$, then any solution of the equation (\ref{cyclradon}) is of the form
\begin{align*} f(t) = \alpha\cdot\mathrm{sm}(\varphi'(t),\varphi'(0)) + \beta\cdot\mathrm{cm}(\varphi(t),\varphi(0)),
\end{align*}
for some numbers $\alpha,\beta \in \mathbb{R}$ which depend on the initial conditions.
\end{teo}
\begin{proof} Notice first that $u_0(t) := \mathrm{sm}(\varphi'(t),\varphi'(0)) = \omega(\varphi'(t),\varphi(0))$. Also, it is clear that $\varphi''(t) = -\delta(t)\varphi(t)$. Hence
\begin{align*} \left(\frac{u_0'(t)}{\delta(t)}\right)' = \left(\frac{\omega(\varphi''(t),\varphi'(0))}{\delta(t)}\right)' = -\omega(\varphi(t),\varphi'(0))' = -\omega(\varphi'(t),\varphi'(0)) = -u_0(t),
\end{align*}
and hence $u_0(t)$ is a solution of (\ref{cyclradon}) with $\lambda = 1$. Now define $u_1(t):=\mathrm{cm}(\varphi(t),\varphi(0))$. Since it is clear that $b(\varphi(t)) = \varphi'(t)$, we get
\begin{align*} \left(\frac{u_1'(t)}{\delta(t)}\right)' = \left(\frac{\omega(\varphi(0),\varphi''(t))}{\delta(t)}\right)' = -\omega(\varphi(0),\varphi(t))' = -\omega(\varphi(0),\varphi'(t)) = -u_1(t),
\end{align*}
from which $u_1$ is also a solution of (\ref{cyclradon}) with $\lambda = 1$. The fact that any solution has to be a linear combination of $u_0$ and $u_1$ comes from the theory of standard ordinary differential equations.

\end{proof}

\section{An approach to Hill's equation} \label{hillsec}

A \emph{Hill equation} is a differential equation of the form
\begin{align}\label{hill1} u''(t)+ f(t)\cdot u(t) = 0,
\end{align}
where $f(t)$ is a periodic function. In the paper \cite{Pet-Bar}, Petty and Barry studied this equation from the viewpoint of Minkowski geometry. In this section, we show that this equation can be obtained as the Minkowskian cycloids equation for $\lambda = 1$. We also argue that this is affine invariant, in the sense that geometries whose unit balls are affinely equivalent yield the same equation. Our last task is to study geometric properties that can be derived from the solutions.

Throughout this section, we assume that $\psi(t):\mathbb{R}\,\mathrm{mod}\,l(\partial B_a)\rightarrow(\mathbb{R}^2,||\cdot||)$ is a parametrization of the unit anti-circle by the arc-length \textbf{of the original norm}. Notice that this is also a parametrization by twice the area of the sectors, meaning that the area of the sector of the unit anti-ball swept from $\psi(t_1)$ to $\psi(t_2)$ equals $t_2 - t_1$. The dual parametrization of the unit circle is the map $\varphi:\mathbb{R}\,\mathrm{mod}\,l(\partial B_a)\rightarrow (\mathbb{R}^2,||\cdot||)$ given as
\begin{align}\label{dualparam1} \varphi(t) = \frac{\psi'(t)}{\omega(\psi(t),\psi'(t))} = \psi'(t),
\end{align}
where we notice carefully that $\omega(\psi(t),\psi'(t)) = 1$ for any $t$. Indeed, since $\psi(t) \dashv_B^{\,a}\psi'(t)$, we get that $\psi'(t) \dashv_B \psi(t)$, and hence
\begin{align*} \omega(\psi(t),\psi'(t)) = ||\psi'(t)||\cdot||\psi(t)||_a = 1.
\end{align*}
It is easy to see that these parametrizations are dual in the sense that the following equality holds:
\begin{align}\label{dualparam2} \psi(t) = -\frac{\varphi'(t)}{\omega(\varphi(t),\varphi'(t))},
\end{align}
for each $t \in \mathbb{R}\,\mathrm{mod}\,l(\partial B_a)$. Fixing these parametrizations, the Minkowskian cycloids equation (\ref{cycleq}) with $\lambda = 1$ can be written as
\begin{align}\label{hilleq} u''(t) + \omega(\varphi(t),\varphi'(t))\cdot u(t) = 0,
\end{align}
which is a Hill equation. Using solutions for a Hill equation, Petty and Barry constructed a curve which they called \emph{indicatrix}. We will go the inverse path, showing that the unit anti-circle yields two linearly independent solutions of (\ref{hilleq}), and hence all solutions (by linear combinations).
\begin{prop}\label{solutions} Decompose $\psi(t)$ in the basis $\{\psi(0),\psi'(0)\}$ as
\begin{align*} \psi(t) = \omega(\psi(t),\psi'(0))\psi(0) - \omega(\psi(t),\psi(0))\psi'(0).
\end{align*}
Then the functions $u_1(t) := \omega(\psi(t),\psi'(0))$ and $u_2(t) := -\omega(\psi(t),\psi(0))$ are linearly independent solutions of (\ref{hilleq}) defined in $\mathbb{R}\,\mathrm{mod}\,l(\partial B_a)$.
\end{prop}
\begin{proof} For $u_1$, notice that (\ref{dualparam1}) and (\ref{dualparam2}) give
\begin{align*} u_1''(t) + \omega(\varphi(t),\varphi'(t))\cdot u_1(t) = \omega(\varphi'(t),\varphi(0)) + \omega(\varphi(t),\varphi'(t))\cdot\omega\left(\frac{-\varphi'(t)}{\omega(\varphi(t),\varphi'(t))},\varphi(0)\right),
\end{align*}
and the latter is clearly equal to $0$. For $u_2$ we simply perform a similar calculation.

\end{proof}

\begin{remark} Observe that the unit anti-ball of a normed plane can be constructed from the solutions of the Hill equation (\ref{hill1}) in case we know \emph{a priori} that the function $f(t)$ is given as $f(t) = \omega(\varphi(t),\varphi'(t))$ for some periodic parametrization $\varphi(t)$ of the boundary of a centrally symmetric convex body (which is the unit circle of the norm).
\end{remark}

Notice carefully that the solutions $u_1,u_2:\mathbb{R}\,\mathrm{mod}\,l(\partial B_a)$ are the unique solutions with initial conditions $u_1(0) = 1$, $u_1'(0) = 0$ and $u_2(0) = 1$, $u_2'(0) = 1$. Moreover, the Wronksian $W(u,v) := u'\cdot v - u\cdot v'$ is constant for any solutions $u,v$, as the following calculation shows:
\begin{align*} W(u,v)'(t) = u\cdot v'' - u''\cdot v = u\cdot(-f\cdot v) - (-f\cdot u)\cdot v = 0.
\end{align*}
Hence the Wronksian only depends on the initial conditions, and for the particular solutions $u_1,u_2$ we have $W(u_1,u_2) = 1$. The geometric interpretation of that is given in the next theorem.
\begin{teo} Up to an affine transformation, the unit anti-circle $\partial B_a$ is obtained from any pair of linearly independent solutions of (\ref{hilleq}). In particular, the planar Minkowski geometry is given intrinsically from the respective differential equation.
\end{teo}
\begin{proof} We claim that, for any linearly independent solutions $u,v$ of (\ref{hilleq}) and any linearly independent vectors $x,y \in \mathbb{R}^2$, the curve
\begin{align*} \gamma(t) = u(t)\cdot x + v(t)\cdot y
\end{align*}
is an affine image of the unit anti-circle $\psi(t)$. Indeed, if $u(0) = c_1$, $u'(0) = c_2$, and $v(0) = d_1$, $v'(0) = d_2$, then we may write
\begin{align*} u(t) = c_1\cdot u_1(t) + c_2\cdot u_2(t), \ \mathrm{and} \\
v(t) = d_1\cdot u_1(t) + d_2\cdot u_2(t).
\end{align*}
Hence it is clear that $\gamma(t) = A(\psi(t))$, where $A:\mathbb{R}^2\rightarrow\mathbb{R}^2$ is the linear transformation such that
\begin{align*} A(\psi(0)) = c_1\cdot x + d_1\cdot y, \ \mathrm{and} \\
A(\psi'(0)) = c_2\cdot x + d_2\cdot y.
\end{align*}
Here we recall that $\psi(t) = u_1(t)\cdot\psi(0) + u_2(t)\cdot\psi'(0)$. Notice that since $x$ and $y$ are linearly independent (as vectors), and $u_1,u_2$ are also linearly independent (as functions), it follows that $A$ is injective.

\end{proof}

\begin{remark} Of course, the geometry given by a fixed norm in a vector space is not affine invariant. However, Minkowski geometry is affine invariant in the following sense: if $K\subseteq \mathbb{R}^n$ is a centrally symmetric convex body, and $A:\mathbb{R}^n\rightarrow\mathbb{R}^n$ is an injective linear map, then the Minkowski geometries whose unit balls are $K$ and $A(K)$ are essentially the same.
\end{remark}

Conversely, it is worth mentioning that the equation (\ref{hilleq}) is uniquely determined by a given norm up to an affine transformation.

\begin{teo} Two norms whose unit balls are linearly equivalent yield the same Hill equation (\ref{hilleq}).
\end{teo}
\begin{proof} As usual, let $\psi$ be a parametrization of the unit anti-circle $\partial B_a$ by arc-length in the norm, and let $A$ be a injective linear map. Denoting by $||\cdot||_{\bar{a}}$ the norm whose unit ball is $A(B_a)$, observe that $\bar{\psi} = A\circ\psi$ is clearly a parametrization of $A(\partial B_a)$, but not necessarily by arc-length in the dual norm of $||\cdot||_{\bar{a}}$. Hence we have to obtain the associated Hill equation as in (\ref{cycleq}). The dual parametrization $\bar{\varphi}$ of $\bar{\psi}$ is given by
\begin{align*} \bar{\varphi}(t) = -\frac{A\psi'(t)}{\omega(A\psi(t),A\psi'(t))} = -\frac{A\psi'(t)}{\mathrm{det}A} = -\frac{A\varphi(t)}{\mathrm{det}A},
\end{align*}
from which
\begin{align*} \omega(\bar{\varphi}(t),\bar{\varphi}'(t)) = \frac{1}{(\mathrm{det}A)^2}\cdot\omega(A\varphi(t),A\varphi'(t)) = \frac{1}{\mathrm{det}A}\cdot\omega(\varphi(t),\varphi'(t))
\end{align*}
follows. Since $\omega(A\psi(t),A\psi'(t)) = \mathrm{det}A\cdot\omega(\psi(t),\psi'(t))$, we have that the equation (\ref{cycleq}), with the parametrizations $\bar{\varphi}$ and $\bar{\psi}$ and $\lambda = 1$, reads precisely
\begin{align*} u''(t) + \omega(\varphi(t),\varphi'(t))\cdot u(t) = 0,
\end{align*}
and this concludes the proof.

\end{proof}

\begin{coro} The Minkowskian sine and cosine functions determine a Radon plane uniquely, up to an affine transformation.
\end{coro}
\begin{proof} From Theorem \ref{lambda1}, the Minkowskian sine and cosine functions are linearly independent solutions of (\ref{hilleq}). Hence the result comes immediately from the theorem above.

\end{proof}

\begin{remark} This corollary is highly intuitive, and possibly one can find a purely geometric proof. However, we opted for stating it here since this is easily obtained as an application of the fact that the geometry of the normed plane is intrinsically related to the corresponding Hill equation.
\end{remark}

Based on the last theorems, from now on we use the expression \emph{geometry associated to the equation (\ref{hilleq})} to refer to the geometries given by the norm and anti-norm constructed (as above) from its solutions. Now we investigate which geometric properties can be derived ``analytically" from the solutions.

As usual, we let $\psi$ be a parametrization of the unit anti-circle by arc-length in the norm, and we denote by $\varphi$ the dual parametrization of the unit circle. The associated Hill equation is given as in (\ref{hilleq}). We let $m:\mathbb{R}\,\mathrm{mod}\,l(\partial B_a)\rightarrow \mathbb{R}\,\mathrm{mod}\,l(\partial B_a)$ be the function such that
\begin{align*} \psi(t+m(t)) = \frac{\psi'(t)}{||\psi'(t)||_a}.
\end{align*}
Geometrically, $m(t)$ is the number such that $\psi(t+m(t))$ is a vector of the unit anti-circle which is parallel to $\psi'(t)$ and has also the same orientation (see Figure \ref{functionm}). In the next proposition we prove that the function $m(t)$ is affine invariant, and hence can be regarded as an intrinsic object of the Minkowski geometry associated to the equation (\ref{hilleq}).
\begin{figure}[h]
\centering
\includegraphics{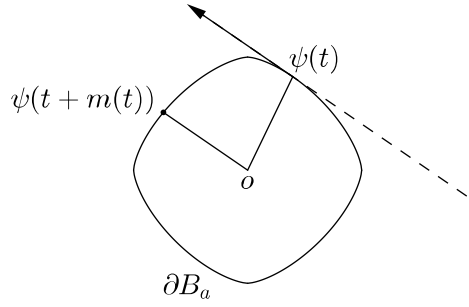}
\caption{The tangent line to $\partial B_a$ at $\psi(t)$ is parallel to $\psi(t+m(t))$.}
\label{functionm}
\end{figure}

\begin{prop} Let $u,v$ be linearly independent solutions of (\ref{hilleq}), and let $x,y \in \mathbb{R}^2$ be linearly independent vectors such that $\omega(x,y) = W(u,v)^{-1}$. If $\bar{m}:\mathbb{R}\,\mathrm{mod}\,l(\partial B_a)\rightarrow\mathbb{R}\,\mathrm{mod}\,l(\partial B_a)$ is the function such that
\begin{align*} \gamma(t+\bar{m}(t)) = \frac{\gamma'(t)}{||\gamma'(t)||_{\bar{a}}},
\end{align*}
where $\gamma(t) = u(t)\cdot x + v(t)\cdot y$ and $||\cdot||_{\bar{a}}$ is the norm induced with $\gamma$ as unit circle, then $\bar{m}(t) = m(t)$ for any $t \in \mathbb{R}\,\mathrm{mod}\,l(\partial B_a)$.
\end{prop}
\begin{proof} The choice of vectors $x,y \in \mathbb{R}^2$ being such that $\omega(x,y) = W(u,v)^{-1}$ guarantees that $\gamma$ is a parametrization of the unit anti-circle by arc-length in the norm. If this hypothesis is dropped, then we get a parametrization by a constant multiple of the arc-length in the norm. Indeed, we have
\begin{align*} \omega(\gamma(t),\gamma'(t)) = \omega(x,y)\cdot(u(t)\cdot v'(t) - u'(t)\cdot v(t)) = \omega(x,y)\cdot W(u,v) = 1,
\end{align*}
and hence the dual parametrization of the unit circle is given as
\begin{align*} \sigma(t) = \frac{\gamma'(t)}{\omega(\gamma(t),\gamma'(t))} = \gamma'(t),
\end{align*}
from which it follows that the length of $\gamma'(t)$ in the norm equals $1$. Hence we geometrically interpret $\bar{m}(t)$ as the length (in the norm) traveled from $\gamma(t)$ to the first vector $\gamma(t+\bar{m}(t))$ which is parallel to $\gamma'(t)$ and has also the same orientation. Since the concepts involved are all invariant under the affine transformation $A$ which carries $\psi(t)$ to $\gamma(t)$, we get that
\begin{align*} A\psi(t+m(t)) = \gamma(t+\bar{m}(t)),
\end{align*}
for each $t \in \mathbb{R}\,\mathrm{mod}\,l(\partial B_a)$. It follows that $t + m(t) = t + \bar{m}(t)$ for every $t$, and hence we have the desired equality.

\end{proof}

As a consequence of this proposition, it follows that $m(t)$ can be defined (and is the same) for any linearly independent solutions $u$ and $v$. One just has to construct the unit anti-circle using linearly independent vectors $x,y \in \mathbb{R}^2$ such that $\omega(x,y) = W(u,v)^{-1}$.

\begin{teo} Let $m(t)$ be the function defined above, associated to the equation (\ref{hilleq}). Then we have:\\

\noindent\textbf{\emph{i.}} The equation comes from a Radon plane if and only if
\begin{align*}\xi(t) := u(t)\cdot v(t+m(t)) - u(t+m(t))\cdot v(t)
\end{align*}
is a constant value for $t \in \mathbb{R}\,\mathrm{mod}\,l(\partial B_a)$, where $u$ and $v$ are any linearly independent solutions. \\

\noindent\textbf{\emph{ii.}} The geometry associated to the equation is Euclidean if and only if $m(t)$ is constant.
\end{teo}
\begin{proof} To prove the first claim, we choose $x,y \in \mathbb{R}^2$ such that $\omega(x,y) = W(u,v)^{-1}$, let $\gamma(t) = u(t)\cdot x + v(t)\cdot y$ be a parametrization of the unit anti-circle by arc-length in the anti-norm, and notice that
\begin{align*} \omega(\gamma(t),\gamma(t+m(t))) = u(t)\cdot v(t+m(t)) - u(t+m(t))\cdot v(t) = \xi(t).
\end{align*}
Since $\gamma(t) \dashv_B^{\bar{a}} \gamma(t+m(t))$, where $\dashv_B^{\bar{a}}$ denotes the Birkhoff orthogonality relation of the norm whose unit circle is $\gamma(t)$, we get that $||\cdot||_{\bar{a}}$ is a Radon norm if $\xi$ is constant, as a consequence of \cite[Proposition 4.1]{Ba-Ma-Tei2}.

Now we prove the second claim. For that sake, recall that the arc-length parametrization of the unit anti-circle in the norm is also a parametrization by twice the areas of the sectors. Hence, if $m$ is constant, then any pair of Birkhoff orthogonal diameters divides the unit anti-circle into four pieces of the same area. This characterizes the planar Euclidean geometry, as it is shown in \cite{alonso2}.

\end{proof}

Actually, the first claim of the theorem can be improved. The solutions for Hill equations associated to Radon planes are characterized by a certain property of their derivatives. In what follows, we recall that we are always assuming that in a Radon plane the symplectic form $\omega$ is always rescaled and so the norm and the anti-norm coincide.

\begin{teo} The equation (\ref{hilleq}) is associated to a Radon plane if and only if there exists a function $g:\mathbb{R}\,\mathrm{mod}\,l(\partial B_a)\rightarrow\mathbb{R}\,\mathrm{mod}\,l(\partial B_a)$ such that the equality
\begin{align*} u'(t) = u(t+g(t))
\end{align*}
holds for any of its solutions.
\end{teo}
\begin{proof} Assume first that the associated norm is Radon. Let $u$ be any solution, and let $v$ be a linearly independent solution. If we fix independent vectors $x,y \in \mathbb{R}^2$, then $\psi(t) = u(t)\cdot x + v(t)\cdot y$ is a parametrization of the unit anti-circle by arc-length in the norm. Hence $\varphi(t) = \psi'(t)$ is a parametrization of the unit circle. Since the unit circle and unit anti-circle coincide, we have immediately from the definition of the function $m$ that
\begin{align*} \varphi(t) = u'(t)\cdot x + v'(t)\cdot y = u(t+m(t))\cdot x + v(t+m(t))\cdot y.
\end{align*}
It follows that $u'(t) = u(t+m(t))$ for any $t \in \mathbb{R}\,\mathrm{mod}\,l(\partial B_a)$. Since $u$ is an arbitrary solution, we have the ``if" part.

For the converse, we simply notice that if such function $g$ exists, then for any $t$ we have the equality
\begin{align*} \varphi(t) = u'(t)\cdot x + v'(t)\cdot y = u(t+g(t))\cdot x + v(t+g(t))\cdot y = \psi(t+g(t)),
\end{align*}
and hence the unit circle and the unit anti-circle coincide. Consequently, the plane is Radon.

\end{proof}

\begin{remark} In the Euclidean case, the standard sine and cosine functions are solutions to the associated Hill equation, and they are translates of their derivatives by the constant $\pi/2$. By the previous theorem, we have something similar for a Radon plane, but the constant is replaced by a function, which nevertheless is universal for all the solutions of the associated Hill equation.
\end{remark}

We have yet another characterization of Radon planes by means of the function $m$. In what follows, we regard $m(t)$ as twice the area of the sector of the unit anti-circle swept from $\psi(t)$ to $\psi(t+m(t))$. This is a little abuse of our definition, since we are assuming that $t \in \mathbb{R}\,\mathrm{mod}\,l(\partial B_a)$.

\begin{teo} The geometry associated to the equation (\ref{hilleq}) is Radon if and only if for the function $m(t)$ defined above we have that
\begin{align*} m(t) + m(t+m(t))
\end{align*}
is constant for $t \in \mathbb{R}\,\mathrm{mod}\,l(\partial B_a)$. In this case, the constant value of the expression is $l(\partial B_a)/2$.
\end{teo}
\begin{proof} If the plane is Radon, then Birkhoff orthogonality is symmetric, and hence
\begin{align*} \psi(t+m(t)+m(t+m(t))) = -\psi(t),
\end{align*}
from where we get immediately that $m(t) + m(t+m(t)) = l(\partial B_a)/2$. For the converse, notice that if $m(t) + m(t+m(t))$ is constant, then the existence of \emph{conjugate diameters} (which are diameters whose directions are mutually Birkhoff orthogonal, see \cite{martini1}) guarantees that it equals $l(\partial B_a)/2$. From the orthogonality relations we obtain
\begin{align*} \psi(t) \dashv_B^a \psi(t + m(t)) \dashv_B^a \psi(t+m(t)+m(t+m(t))) = \psi(t+l(\partial B_a)/2) = -\psi(t),
\end{align*}
for any $t \in \mathbb{R}\,\mathrm{mod}\,l(\partial B_a)$. It follows that Birkhoff orthogonality is a symmetric relation, and hence the geometry is Radon.

\end{proof}

\begin{remark} The geometric meaning of the theorem above is that, in a Radon plane, the symmetry of Birkhoff orthogonality guarantees that applying $m$ iteratively twice is the same as doing a reflection through the origin. This does not happen in the case where the plane is not Radon, as Figure \ref{twicem} below illustrates.
\end{remark}

\begin{figure}[h]
\centering
\includegraphics{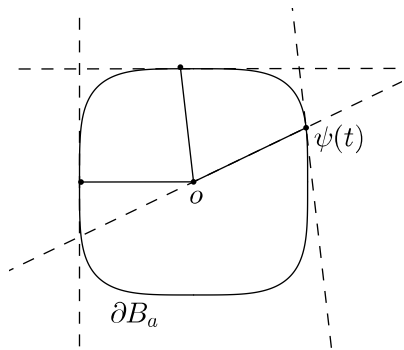}
\caption{$\psi(t+m(t)+m(t+m(t))) \neq -\psi(t)$.}
\label{twicem}
\end{figure}

\section{When is a Hill equation associated to a Minkowski geometry?}\label{condhill}

In this section we study the natural problem under which conditions a given Hill equation (\ref{hill1}) is associated to the geometry of some normed plane. In this direction, we have the following theorem.

\begin{teo}\label{condition} Let $f:\mathbb{R}\,\mathrm{mod}\,2c\rightarrow\mathbb{R}$ be a smooth, strictly positive $2c$-periodic function, where $c > 0$. If the Hill equation
\begin{align*} u''(t) + f(t)\cdot u(t)= 0
\end{align*}
has two independent $c$-antiperiodic solutions, and if the equation
\begin{align*} u''(t) + \lambda\cdot f(t)\cdot u(t) = 0
\end{align*}
has no $c$-antiperiodic solutions for any $0 < \lambda  < 1$, then the Hill equation is associated to the geometry of some normed plane.  
\end{teo}
\begin{proof} Let $u,v$ be two ($c$-antiperiodic) solutions and, for simplicity, assume that $W(u,v) = 1$. We choose linearly independent vectors $x,y \in \mathbb{R}^2$ such that $\omega(x,y) = 1$ and construct the curve
\begin{align*} \psi(t) = u(t)\cdot x + v(t) \cdot y.
\end{align*}
Since $u$ and $v$ are $c$-antiperiodic, this curve is closed and centrally symmetric. Also, we notice that $\psi$ is locally strictly convex since $\omega(\psi(t),\psi'(t)) = 1$ for all $t$. The hypothesis on $\lambda = 1$ being the smaller positive eigenvalue of the associated Sturm-Liouville problem guarantees that both $u$ and $v$ have a unique zero in the interval $[0,c)$, as it is proved in \cite[Chapter 8, Theorem 3.1]{coddington}. Hence $\psi$ does not have any self-intersection before it closes, and thus $\psi$ is strictly convex.

Define now
\begin{align*} \varphi(t) = \psi'(t) = u'(t)\cdot x + v'(t)\cdot y.
\end{align*}
This curve is also closed and centrally symmetric, and it has no self-intersection before it closes. Also,
\begin{align*} \omega(\varphi(t),\varphi'(t)) = u'(t)\cdot v''(t) - u''(t)\cdot v'(t) = f(t)\cdot u(t)\cdot v'(t) - f(t)\cdot u(t)\cdot v'(t) = \\ = f(t)\cdot(u(t)\cdot v'(t) - u'(t)\cdot v(t)) = f(t)\cdot W(u,v)(t) = f(t) > 0.
\end{align*}
Hence $\varphi$ is strictly convex, and moreover $\omega(\varphi(t),\varphi'(t)) = f(t)$. It follows that the Hill equation is associated to the geometry given by $\varphi$ as unit circle and $\psi$ as unit anti-circle. This finishes the proof.

\end{proof}

\begin{remark}\label{remarkdouble} In \cite{Cra-Tei-Ba1} it was proved that if $\lambda = 1$ is a double eigenvalue of the Sturm-Liouville problem, then the \emph{indicatrix} (constructed from the solutions) is locally strictly convex. What we proved is that if we add the hypothesis of $\lambda = 1$ being the \textbf{smallest} positive eigenvalue, then the indicatrix is indeed the boundary of a strictly convex body, and hence we have the geometry of a normed plane. Moreover, what really matters is that the first non-zero eigenvalue $\lambda_1$ of (\ref{hill1}) is double. Indeed, in this case, the Hill equation
\begin{align*} u''(t) + \lambda_1\cdot f(t)\cdot u(t) = 0
\end{align*}
is such that its smallest non-zero eigenvalue is $1$, and this is a double eigenvalue. From Theorem \ref{condition} it follows that the Hill equation above is associated to the geometry of some normed plane. 
\end{remark}

A very natural problem is to decide whether a given Hill equation (\ref{hill1}) is associated to some Minkowski geometry without looking at its solutions, but only at the function $f$. Next we interpret this function as a certain curvature function to derive a necessary condition to what happens. For that, we need a concept of curvature other than circular curvature (introduced in Section \ref{basic}), which we define now. Let $\gamma(s)$ be a smooth curve parametrized by arc-length in a normed plane whose unit circle $\partial B$ is endowed with a parametrization $\varphi(\tau)$ by twice the area of the sectors (and hence by the arc-length in the anti-norm). Let $\tau(s)$ be a function such that $\gamma'(s) = \varphi(\tau(s))$. Then the \emph{Minkowski curvature} of $\gamma$ at $\gamma(s)$ is the number
\begin{align*} k_m(s) := \dot{\tau}(s),
\end{align*}
where the dot denotes the derivative with respect to the parameter $s$. We refer the reader to \cite{Ba-Ma-Sho} for more details on this curvature type, and also to other curvature concepts in normed planes.

\begin{teo}\label{funccurvature} If the Hill equation
\begin{align*} u''(t) + f(t)\cdot u(t) = 0
\end{align*}
induces a Minkowski geometry in the plane, then $f$ is the inverse of the Minkowski curvature of $\partial B$ in the anti-norm. 
\end{teo} 
\begin{proof} As usual, let $\psi(t)$ be a parametrization of the unit anti-circle by the length in the original norm (hence a parametrization by twice the area of the sectors), and let $\varphi(t) = -\psi'(t)$ be the dual parametrization of the unit circle. Hence the associated Hill equation is given as
\begin{align*} u''(t) + \omega(\varphi(t),\varphi'(t))\cdot u(t) = 0,
\end{align*}
and we claim that the function $f(t) = \omega(\varphi(t),\varphi(t))$ is the inverse of the Minkowski curvature of the unit circle $\partial B$ in the anti-norm (or the \emph{arc-length curvature} of $\partial B$ in the norm, see \cite{Ba-Ma-Sho}). First of all, let $\varphi(s) := \varphi(t(s))$ be a parametrization of the unit circle by arc-length in the anti-norm, where $t(s)$ is an increasing function. Differentiating and evaluating in the anti-norm one gets
\begin{align*} 1 = ||\dot{\varphi}(s)||_a = \dot{t}(s)\cdot||\varphi'(t)||_a = \dot{t}(s)\cdot\omega(\varphi(t),\varphi'(t)),
\end{align*}
and hence $\dot{t}(s) = \omega(\varphi(t),\varphi'(t))^{-1}$. Now, notice that
\begin{align*} \psi(t(s)) = \frac{\varphi'(t)}{\omega(\varphi(t),\varphi'(t))} = \dot{t}(s)\cdot\varphi'(t) = \dot{\varphi}(s),
\end{align*}
from which the Minkowski curvature $k_m(s)$ of $\varphi$ at $\varphi(t(s))$ calculated in the anti-norm is given by $k_m(s) = \dot{t}(s) = \omega(\varphi(t),\varphi'(t))^{-1}$. Hence our claim is proved. 

\end{proof}

As a consequence of these last results we obtain that an analytic property of (\ref{hill1}) can be rewritten in geometric terms. The proof is straightforward. 

\begin{coro} Let $f:\mathbb{R}\,\mathrm{mod}\,c \rightarrow \mathbb{R}$ be a smooth, strictly positive, $c$-periodic function. Then the first positive eigenvalue $\lambda_1$ of the Sturm-Liouville problem
\begin{align*} u''(t) + \lambda\cdot f(t)\cdot u(t) = 0
\end{align*}
is double if and only if $\lambda_1\cdot f(t)$ is the inverse of the Minkowski curvature function of some Minkowski ball calculated in its respective anti-norm.  
\end{coro}

A conjecture in the paper \cite{Cra-Tei-Ba1} refers to the question whether a normed plane is necessarily Euclidean if all of the eigenvalues of the associated Sturm-Liouville problem are double. In this direction, we prove that the Euclidean case is characterized if some non-unit eigenvalue (with antiperiodic boundary conditions) induces the original geometry. 

In what follows, we assume (as usual) that $\psi$ is a parametrization of the unit anti-circle $\partial B_a$ of a given normed plane by arc-length in the norm, and that $\varphi$ is the dual parametrization of the unit circle. Hence the associated Hill equation is, of course, given as in (\ref{hilleq}).

\begin{teo} Assume that there exists a positive number $\lambda \neq 1$ such that the equation
\begin{align}\label{eqlambda} u''(t) + \lambda\cdot\omega(\varphi(t),\varphi'(t))\cdot u(t) = 0
\end{align}
admits a $c$-antiperiodic solution which is a linear reparametrization of a non-zero solution of (\ref{hilleq}). Then the curve $\varphi$ induces an Euclidean norm. 
\end{teo}
\begin{proof} We have to prove that $f(t) := \omega(\varphi(t),\varphi'(t))$ is a constant function. Let $u$ be a non-zero solution of (\ref{hilleq}) and $v(t) = u(\alpha t)$ be a solution of (\ref{eqlambda}), where $\alpha$ is some constant. Then one gets
\begin{align*} \alpha^2\cdot u(\alpha t) = v''(t) = -\lambda\cdot f(t)\cdot v(t).
\end{align*}
On the other hand, we have $u''(\alpha t) = -f(\alpha t)\cdot u(\alpha t) = -f(\alpha t)\cdot v(t)$. This leads to the equality
\begin{align*} \alpha^2\cdot f(\alpha t)\cdot v(t) = \lambda\cdot f(t)\cdot v(t).
\end{align*}
Since $v(t)$ has isolated zeros, it follows from continuity that
\begin{align}\label{fprop} f(\alpha t) = \frac{\lambda}{\alpha^2}\cdot f(t) = \beta\cdot f(t).
\end{align}
From now on we use $\beta = \lambda/\alpha^2$ and prove that $\beta = 1$. Indeed, from continuity and compactness we have that $0 < \min_{\mathbb{R}} f \leq \max_{\mathbb{R}} f < \infty$, where we recall that $f$ is $c$-periodic. Then we have that $f(0) < \beta^n\min_{\mathbb{R}} f$ for $n$ sufficiently large if $\beta > 1$ and $f(0) > \beta^n\max_{\mathbb{R}} f$ for $n$ sufficiently large if $\beta < 1$. If $n$ is large enough such that one of the estimates above holds, then we choose $t_0 \in \mathbb{R}$ and $N \in \mathbb{Z}$ such that $\alpha^nt _0= Nc$. We get
\begin{align*} \beta^nf(t_0) = f(\alpha^nt_0) = f(Nc) = f(0),
\end{align*} 
which is a contradiction, since we obviously have $\beta^n\min_{\mathbb{R}}f \leq \beta^nf(t_0) \leq \beta^n\max_{\mathbb{R}}f$. 

Notice that there are two consequences of the equality $\beta = 1$. The first is that $\alpha \neq 1$, and we also may assume that, without loss of generality, $\alpha > 0$. Also, equality (\ref{fprop}) reads
\begin{align*} f(\alpha t) = f(t),
\end{align*}
for any $t \in \mathbb{R}$. We claim that this condition, together with the $c$-periodicity of $f$, implies that $f$ is constant. Of course, we only need to prove it if $c \neq 0$. Notice that we have
\begin{align*} f(t+\alpha^pkc) = f\left(\alpha^p\left(\frac{t}{\alpha^p}+kc\right)\right) = f\left(\frac{t}{\alpha^p}+kc\right) = f\left(\frac{t}{\alpha^p}\right) = f(t),
\end{align*}
for any $t \in \mathbb{R}$ and any $p,k\in\mathbb{Z}$. It follows that
\begin{align*} f(0) = f(c(k_1\alpha^{p_1}+k_2\alpha^{p_2}+\ldots + k_n\alpha^{p_n}))
\end{align*}
for any $n\in\mathbb{N}$ and any numbers $k_1,\ldots,k_n,p_1,\ldots,p_n \in \mathbb{Z}$. Since
\begin{align*} G := \{c(k_1\alpha^{p_1}+k_2\alpha^{p_2}+\ldots + k_n\alpha^{p_n}):n\in\mathbb{N}, \ k_1,\ldots,k_n,p_1,\ldots,p_n \in \mathbb{Z}\}
\end{align*}
is an additive subgroup of $\mathbb{R}$ with no least positive element, it follows that $G$ is dense in $\mathbb{R}$. From continuity we have that $f$ is constant. This finishes the proof. 

\end{proof}

\end{document}